\documentclass[12pt,letterpaper]{amsart}
\usepackage{amsmath, amsthm, amssymb, xspace, mathrsfs}
\usepackage{geometry}
\usepackage[breaklinks=true]{hyperref}
\usepackage{amsmath,amscd}
\usepackage{tikz-cd}

\theoremstyle{plain}
\newtheorem{theorem}{Theorem}[section]
\newtheorem{conjecture}{Conjecture}

\newtheorem{lemma}{Lemma}[section]
\newtheorem{prop}{Proposition}[section]

\theoremstyle{definition}

\newtheorem{remark}{Remark}[section]
\newtheorem*{acknowledgement}{Acknowledgements}

\newcommand{\mf}[1]{\displaystyle{\mathfrak{#1}}}

\newcommand{\comment}[1]{}

\DeclareMathOperator{\spec}{\ensuremath{Spec}}

\DeclareMathOperator{\Gr}{\ensuremath{gr}}

\DeclareMathOperator{\Sym}{\ensuremath{Sym}}

\DeclareMathOperator{\Frac}{\ensuremath{Frac}}

\DeclareMathOperator{\rank}{\ensuremath{rank}}
\DeclareMathOperator{\ndex}{\ensuremath{index}}
\DeclareMathOperator{\HC}{\ensuremath{HC}}

\begin{document}
\title[the center of the enveloping algebras in large characteristic]{A generalization of Veldkamp's theorem for a class of Lie algebras}
\author{Akaki Tikaradze}
\email{ tikar06@gmail.com}
\address{University of Toledo, Department of Mathematics \& Statistics, 
Toledo, OH 43606, USA}
\begin{abstract}

A classical theorem of Veldkamp describes the center of an enveloping algebra of a Lie algebra of a semi-simple
algebraic group in characteristic $p.$ We generalize this result to a class of Lie algebras with a property that they arise as the reduction modulo
$p\gg 0$ from an algebraic  Lie algebra $\mathfrak{g},$  such that $\mathfrak{g}$ has no nontrivial semi-invariants in $\Sym(\mathfrak{g})$
and $\Sym(\mathfrak{g})^{\mf{g}}$ is a polynomial algebra.

As an application, we  solve the  derived isomorphism problem of enveloping
algebras for the above class of Lie algebras.

\end{abstract}

\maketitle

\section{Introduction}

Understanding the center of the enveloping algebra of a finite dimensional Lie algebra
in characteristic $p>0$ is of fundamental importance in the modular representation theory of Lie algebras.
Many such Lie algebras arise from a base change from a Lie algebra in characteristic 0.
Let $\mathfrak{g}$  be a Lie algebra of an algebraic group over  $S\subset\mathbb{C}$-a finitely generated ring,
let $S\to \bold{k}$ be a base change to a characteristic $p$ field. Then the center of  the enveloping algebra $U(\mathfrak{g}_{\bold{k}})$
 contains two distinguished subalgebras: the $p$-center $Z_p$ which is generated by elements of the form $g^p-g^{[p]}, g\in \mathfrak{g}_{\bold{k}},$
 and the image of $Z(U(\mathfrak{g}))$ in $Z(U(\mathfrak{g}_{\bold{k}})),$ to be denoted by $Z_{\text{HC}},$ (the Harish-Chandra part of the center).
Thus, it is a very natural to ask whether  for large enough $p$ the center of  $U(\mathfrak{g}_{\bold{k}})$
is generated by $Z_p$ and $Z_{\text{HC}}$ (stated as a conjecture in \cite{K}). 
 Unfortunately the answer is  already negative 
for indecomposable 3-dimensional solvable Lie algebras (see Remark \ref{remark}). For the case of a semi-simple $\mathfrak{g}$, the answer is positive
as follows from the classical theorem of Veldkamp \cite{V}.

The main result of this paper can be seen as a generalization of  Veldkamp's theorem for a class of Lie algebras. 

\begin{theorem}\label{center}
Let $\mathfrak{g}$ be an algebraic Lie algebra over a finitely generated ring $S\subset\mathbb{C}$,
such that $\Sym \mathfrak{g}$ has no nontrivial $\mathfrak{g}$-semi-invariants, and 
$\Sym(\mathfrak{g})^{\mathfrak{g}}=S[f_1,\cdots, f_n]$ is a polynomial algebra.
Let $g_i\in Z(U(\mathfrak{g}))$ be the symmetrization of $f_i, 1\leq i\leq n.$
Then for all primes $p\gg 0$ and a base change $S\to \bold{k}$ to a field of characteristic $p,$ the center
of $U(\mathfrak{g}_{\bold{k}})$  is a 
 free $Z_p$-module
with a basis $\lbrace g^{\alpha}=g_1^{\alpha_1}\cdots g_n^{\alpha_n}, 0\leq\alpha_i<p\rbrace.$ 
Moreover   $Z(U(\mathfrak{g}_{\bold{k}}))$ is a complete intersection ring and
$$Z(U(\mathfrak{g}_{\bold{k}}))\cong Z_p\otimes_{Z_p\cap Z_{\HC}} Z_{\HC}, \quad Z_{\HC}=U(\mathfrak{g})^{G_{\bold{k}}}.$$

\end{theorem}
This result was obtained in \cite{T1} under the addition assumption that
 the coadjoint action of $G$ (algebraic group corresponding to $\mf{g}$) on $\spec(\Sym(\mathfrak{g})/(f_1,\cdots, f_n))\subset \mathfrak{g}^*$
has an open dense orbit.

\begin{remark}\label{remark}
 Let $\mathfrak{g}$ be the following
3-dimensional Lie algebra: $\mathfrak{g}=\mathbb{C}x\oplus \mathbb{C}y\oplus\mathbb{C}z$
with the bracket $$[z,x]=nx, [z,y]=my, [x, y]=0, n, m\in\mathbb{N}.$$
 Then for large enough $p$ and a characteristic $p$ field $\bold{k}$, we have that $Z_{\HC}=\mathbf{k}.$
 Let $1\leq a, b< p$ be such that $na=bm.$ Then $x^ay^{p-b}\in Z(U(\mathfrak{g}_{\bold{k}}))\setminus Z_p.$
Therefore  $Z(U(\mathfrak{g}_{\bold{k}}))$ is not  generated by $Z_p, Z_{\text{HC}}.$ On the other hand, $x, y\in \Sym(\mathfrak{g})$ represent
nontrivial semi-invariants.

\end{remark}

\begin{remark}\label{examples}
Let us discuss known examples of Lie algebras $\mf{g}$ satisfying the  assumption in Theorem \ref{center} 
that $\Sym(\mathfrak{g})$ has no nontrivial $\mathfrak{g}$-semi-invariants and 
$\Sym(\mathfrak{g})^{\mathfrak{g}}$ is a polynomial algebra in detail.
Let $\mathfrak{g}$ be a semi-simple Lie algebra, then for the corresponding multiparameter Takiff  algebras  (also known as truncated multicurrent algebras)
$\mathfrak{g}_m=\mathfrak{g}\otimes_{\mathbb{C}}\mathbb{C}[t_1,\cdots, t_n]/(t_1^{m_1},\cdots, t_n^{m_n})$ it is known that
$\Sym(\mathfrak{g}_m)^{\mathfrak{g}_m}$ is a polynomial algebra  \cite{MS}  (see also \cite{PY}). Since $\mathfrak{g}_m$ is a perfect Lie algebra,
it follows that multi-parameter Takiff algebra satisfy the assumption in Theorem \ref{center}.

There are a lot of results in the literature  whether $L=\mathfrak{g}\ltimes V$ has the property that  $\Sym(L)^L$ is a polynomial algebra, where
$\mathfrak{g}$ is a simple Lie algebra and $V$ its representation (as an incomplete list, see \cite{K1}, \cite{K2}, \cite{P}, \cite{PY}, \cite{Y}). Just few examples to name, the following
perfect algebras satisfy the assumption:
$\mf{sl}_n\ltimes \mathbb{C}^n, \mf{so}_n\ltimes\mathbb{C}^n, \mf{sp}_{2n}\ltimes \mathbb{C}^{2n}.$

Since nilpotent Lie algebras have no nontrivial semi-invariants, any nilpotent Lie algebra $\mathfrak{g}$
with the property that $\Sym (\mathfrak{g})^{\mathfrak{g}}$ is a polynomial algebra satisfies assumptions of Theorem \ref{center}.
The class of such Lie algebras is quite large: all but 3 indecomposable nilpotent Lie algebras of dimension $\leq 7$ \cite{O1}, nilpotent Lie algebras of index at most 2 \cite{O2}, nilradical of a borel subalgebra of a semi-simple Lie algebra [\cite{O2}, Corollary 32].


\end{remark}

The proof of Theorem \ref{center} crucially relies on the proof of the  first Kac-Weisfeler conjecture for very
large primes given in \cite{T2}.

Next we apply Theorem \ref{center} to the isomorphism problem of enveloping algebras-a well-known open problem in ring theory. 
Recall that it asks whether a $\mathbb{C}$-algebra isomorphism between enveloping
algebras of Lie algebras implies an isomorphism of the underlying Lie algebras.
For the background and the detailed discussion of this problem we refer the reader to  the survey article by Usefi  \cite{U}.
Very recently,  the  solution of the isomorphism problem of enveloping algebras for nilpotent Lie algebras was presented in \cite{CPNW},
using techniques of higher category theory.

In analogy with the derived isomorphism problem for rings of differential operators
for smooth affine varieties, in \cite{T1} we considered the following natural generalization.

\begin{conjecture}
Let $\mathfrak{g}, \mathfrak{g'}$ be finite dimensional Lie algebras over $\mathbb{C}.$
If the derived categories of bounded complexes of  $U(\mathfrak{g})$-modules and $ U(\mathfrak{g'})$-modules are equivalent,
then $\mathfrak{g}\cong \mathfrak{g'}.$
\end{conjecture}

In our approach to the above conjecture we follow the well-known blueprint of "dequantization"
by reducing to the modulo large prime $p$ technique, which allows a translation of questions 
about various quantizations to the ones about Poisson algebras, inspired largely by the proof of Belov-kanel an 
Kontsevich of equivalence between the Jacobian and Dixmier conjectures.
Using this approach, we proved  in \cite{T1} that certain class of Lie algebras that includes Frobenius Lie algebras
 are derived invariants
of their enveloping algebras.





As an application of Theorem \ref{center} to the derived equivalence problem 
we have the following.

\begin{theorem}\label{iso}

Let $\mathfrak{g}, \mathfrak{g'}$ be algebraic Lie algebras satisfying assumptions in Theorem \ref{center}. If $U(\mathfrak{g})$
is derived equivalent to $U(\mathfrak{g'})$ then $\mathfrak{g}/Z(\mathfrak{g})\cong \mathfrak{g'}/Z(\mathfrak{g'}).$
\end{theorem}

\section{The first Kac-Weisfeiler conjecture}

In this section we  recall some results associated with the first Kac-Weisfeiler conjecture that are used
in proof of our main results.
Recall that the first Kac-Weisfeiler conjecture asserts that for a $p$-restricted Lie algebra $\mathfrak{g}$
over
$\bold{k}$, the maximal  possible dimension  of an irreducible $\mathfrak{g}$-module is
$p^{\frac{1}{2}(\dim (\mathfrak{g})- \ndex(\mathfrak{g}))}$. Equivalently
the rank of $U(\mathfrak{g})$ over its center 
equals $p^{\dim (\mathfrak{g})-\ndex(\mathfrak{g})}.$

Let $\mathfrak{g}$ be a Lie algebra of an algebraic group $G$ defined over a finitely generated
ring $S\subset \mathbb{C}.$ Then for all $p\gg 0$ and a base change $S\to\mathbf{k}$ to an algebraically
closed field of characteristic $p,$ the first Kac-Weisfeiler conjecture was established for $\mathfrak{g}_{\bold{k}}$ 
by Martin, Stewart, and  Topley. Also proved independently by the author.
Namely we make use of the following result.
\begin{theorem}[\cite{T2}, Theorems 3.8 and 3.9]\label{KW}
Let $\mathfrak{g}$ be an algebraic Lie algebra over a finitely generated ring $S\subset\mathbb{C}.$
Then for all $p\gg 0$ and a base change $S\to\mathbf{k}$ to an algebraically closed field of characteristic $p$,
  the fraction field of $Z(U(\mathfrak{g}_{\bold{k}}))$ 
is generated by the image of the center of $\Frac(U(\mathfrak{g}))$  and  $\Frac(Z_p(\mathfrak{g}_{\bold{k}})).$
Also, the following equality of degrees of field extensions holds:\

$$[\Frac(Z(U(\mathfrak{g}_{\bold{k}}))): \Frac(Z_p(\mathfrak{g}_{\bold{k}}))]=p^{\ndex(\mathfrak{g})}=
[\Frac(\Sym(\mathfrak{g}_{\bold{k}})^{\mathfrak{g}_{\bold{k}}}): \Frac(\Sym(\mathfrak{g}_{\bold{k}})^p)]$$

\end{theorem}

\begin{remark}
The equality 
$$[\Frac(\Sym(\mathfrak{g}_{\bold{k}})^{\mathfrak{g}_{\bold{k}}}): \Frac(\Sym(\mathfrak{g}_{\bold{k}}))^p]=p^{\ndex(\mathfrak{g})}$$
follows from the fact that on the one hand (as proved in [\cite{T2} Theorem 3.8, 3.9])
$$[\Frac(\Gr Z(U(\mathfrak{g}_{\bold{k}})): \Frac(\Sym(\mathfrak{g}_{\bold{k}}))^p]\geq p^{\ndex(\mathfrak{g})},$$
and on the other hand it was proved in \cite{PS} that 
$$[\Frac(\Sym\mathfrak{g}_{\bold{k}}): \Frac(\Sym(\mathfrak{g}_{\bold{k}})^{\mathfrak{g}_{\bold{k}}})]\geq p^{\dim (\mathfrak{g})-\ndex(\mathfrak{g})}.$$

\end{remark}

We also need to recall the following simple result from commutative algebra (for a proof see [\cite{T2} Lemma 3.11]. )

\begin{lemma}\label{rank}
Let $S\subset\mathbb{C}$ be a finitely generated ring.
Let $A$ be a finitely generated commutative algebra over $S$ such that $A_{\mathbb{C}}$ is a domain.
 Let $B\subset A$ be a finitely generated  $S$-subalgebra.
Then for all $p\gg 0$ and a base change $S\to\bold{k}$ to an algebraically closed field $\bold{k}$ of characteristic $p$
the rank of $A_{\bold{k}}$ over $B_{\bold{k}}A_{\bold{k}}^p$ is  $p^{\dim(A)-\dim(B)}.$
\end{lemma}

\section{The proof of Theorem \ref{center}}

We crucially rely on the following result  from Panyushev-Yakimova [\cite{PY}, Remark 1.3] 
(see also [\cite{PY}, Proposition 1.2], [\cite{JS} Proposition 5.2]).

\begin{prop}\label{codim2}
Let $\mathfrak{g}$ be a an algebraic Lie algebra over $\mathbb{C}$ such that 
 $\Sym(\mathfrak{g})$ has no nontrivial $\mathfrak{g}$-semi-invariants.
Let $\mathcal{O}=\Sym (\mathfrak{g})^{\mathfrak{g}}$ be finitely generated.
Put $Y=\spec  \mathcal{O}.$ We have the quotient map $\pi:\mathfrak{g}^*=\spec \Sym(\mathfrak{g})\to Y.$ 
Denote by $Y_{sm}$ the smooth locus of $Y.$
Let $U=\lbrace x\in \mathfrak{g}^*, \pi(x)\in Y_{sm},  d\pi_{x} \text{is onto}\rbrace.$
Then $\mathfrak{g}^*\setminus U$
has  codimension $\geq 2$ in $\mathfrak{g}^*.$
\end{prop}

\begin{proof}[Proof of  Theorem \ref{center}]

Since $\mathfrak{g}$ has no nontrivial semi-invariants in $\Sym (\mathfrak{g})$, it follows easily that
$$(\Frac(\Sym (\mathfrak{g})))^{\mathfrak{g}}=\Frac(\Sym (\mathfrak{g})^{\mathfrak{g}}).$$ 
 Put 
 $$\mathcal{O}=\bold{k}[f_1,\cdots, f_n], \quad A=\Sym(\mathfrak{g}_{\bold{k}}),\quad B=\Sym(\mathfrak{g}_{\bold{k}})^p\mathcal{O},\quad 
 B'=(\Sym \mathfrak{g}_{\bold{k}})^{\mathfrak{g}_{\bold{k}}}.$$ 
Clearly $B\subset B'.$ Let $K=\Frac(A)$ and $K_0=\Frac(B).$

We first claim that $$[K_0:K^p]=p^n.$$
Indeed, this follows at once since the degree of $K$ over $K_0$ is $p^{\dim(\mathfrak{g})-n}$ by Lemma \ref{rank},
and 
$$[K: K^p]=p^{\dim \mathfrak{g}}.$$

Next we argue that $B$ a normal domain.
Indeed, it follows from Proposition \ref{codim2} that for all $p\gg 0$ and a base change $S\to \mathbf{k},$ the compliment of 
$$U_{\bold{k}}=\lbrace x\in \mathfrak{g}_{\bold{k}}^*,  d\pi_{x} \text{ is onto}\rbrace.$$
in ${\mathfrak{g}}_{\bold{k}}^*$ has codimension $\geq 2.$ 
We claim that the multiplication map $\phi$ induces an isomorphism 
$$ \phi:A^p\otimes_{\mathcal{O}^p} \mathcal{O}\cong B.$$
Indeed, $\phi$ is a surjective map from a free $A^p$-module of rank $p^n$  onto a $A^p$-module of rank $p^n.$ Thus $\phi$ must be an isomorphism.
 If $I\in U_{\bold{k}},$ put $I'=I\cap\mathcal{O}.$ It follows easily that  $A^p_{I}\otimes_{\mathcal{O}_{I}^p} \mathcal{O}_{I'}$ is a regular ring.
Hence $B$ is a Cohen-Macaulay ring regular in codimension 1. Therefore it is a normal domain by Serre's criterion of normality.
     
     It follows that $B'$ and $B$ have the same field of fractions
as they are both extensions of degree $p^n$ of $A^p$ by Theorem \ref{KW} and
$B\subset B'$. Thus, normality of $B$ yields that $B=B'.$ Hence 
$$\Sym (\mathfrak{g}_{\bold{k}})^{\mathfrak{g}_{\bold{k}}}=\Sym (\mathfrak{g}_{\bold{k}})^p[f_1,\cdots,f_m],\quad 
Z(U(\mathfrak{g}_{\bold{k}}))=Z_p[g_1,\cdots, g_m].$$

Next we show that $B$ is a free $A^p$-module with basis $\lbrace f^{\alpha}=\prod_{i=1}^n f_i^{\alpha_i}, \alpha_i<p\rbrace.$
Indeed, since $[K_0:K^p]=p^n,$ it follows that $\lbrace f^{\alpha}, \alpha_i<p\rbrace$
are linearly independent over $K^p.$ In particular they form a basis of $B$ over $A^p.$
Hence, $Z(U(\mathfrak{g}_{\bold{k}}))$ is a free module over $Z_p$ with a basis $\lbrace g^{\alpha}, \alpha_i<p\rbrace$ as desired.
Thus we obtain that 
$$ \Gr Z(U(\mathfrak{g}_{\bold{k}}))=B=(\Sym (\mathfrak{g}_{\bold{k}}))^{\mathfrak{g}_\bold{k}}.$$

We have the natural surjective homomorphism (that restricts to the Frobenius homomorphism on $A$ and sends $x_i$ to $f_i^p$)
$$A[x_1,\cdots, x_n]/(x_1^p-f_1,\cdots, x_n^p-f_n)\to B,$$
which must be an isomorphism since both rings are free $A^p$-modules of rank $p^n.$
 Since $(x_1^p-f_1,\cdots, x_n^p-f_n)$
is a regular sequence in $A[x_1,\cdots, x_n],$ it follows that $B$ is a complete intersection.
Therefore $Z(U(\mathfrak{g}_{\bold{k}}))$ is also a complete intersection ring.

Our next goal is show that $\Sym (\mathfrak{g}_{\bold{k}})^{G_{\bold{k}}}=\bold{k}[f_1,\cdots, f_n],$
which implies $U(\mathfrak{g})^{G_{\bold{k}}}=Z_{\HC}.$
Let $x\in \Sym (\mathfrak{g}_{\bold{k}})^{G_{\bold{k}}}\subset B.$ 
As $B$ is a free $\Sym (\mathfrak{g}_{\bold{k}})^p$ module with basis $\lbrace f^{\alpha}=\prod_{i=1}^n f_i^{\alpha_i}, \alpha_i<p\rbrace,$
we may write
$$x=\sum_{\alpha}x_{\alpha}f^{\alpha}, \quad x_{\alpha}\in (\Sym \mathfrak{g}_{\bold{k}})^p, \alpha_i<p.$$ 
Then $$x_{\alpha}\in (\Sym (\mathfrak{g}_{\bold{k}})^{G_{\bold{k}}})^p.$$
Replacing $x$ by $x_{\alpha}^{\frac{1}{p}}$ and continuing in this manner, we obtain that 
$$x\in \bigcap_{n}(\Sym (\mathfrak{g}_{\bold{k}}))^{p^n}\mathcal{O}=\mathcal{O}.$$
We also get that 
$$Z_p\cap Z_{\HC}=Z_p^{G_{\bold{k}}}.$$
So $$\Gr(Z_p\cap Z_{\HC}) =\bold{k}[f_1^p,\cdots, f_n^p].$$
Hence $Z_{\HC}$ is a free 
$Z_p\cap Z_{\HC}$-module of rank $p^n.$ Therefore the natural ring homomorphism
$$Z_p\otimes_{Z_p\cap Z_{\HC}}Z_{\HC}\to Z(U(\mathfrak{g}_{\bold{k}}))$$
is a surjective homomorphism of free $Z_p$-modules of rank $p^{n}.$ Thus it is an isomorphism as desired.


\end{proof}

\section{Applications to rigidity of enveloping algebras}

Recall that given an associative flat $\mathbb{Z}$-algebra $R$ and a prime number $p,$
 then the center $Z(R/pR)$ of its reduction modulo $p$  acquires the natural Poisson bracket,
defined as follows. Given  central elements $a, b\in Z(R/pR)$, let $z, w\in R$ be their lifts respectively. 
Then the Poisson bracket $\lbrace a, b \rbrace$ is defined to be $$\frac{1}{p}[z, w] \mod p\in Z(R/pR).$$
 
\noindent  In particular, given a a Lie algebra $\mathfrak{g}$ over $S\subset\mathbb{C}$--a finitely generated ring
such that $\mathfrak{g}$ is a finite free $S$-module, then for a prime $p>0,$ the center 
$Z(U(\mathfrak{g}_p))$
of the enveloping algebra of $\mathfrak{g}_{p}=\mathfrak{g}/p\mathfrak{g}$ is equipped with the natural $S/pS$-linear Poisson bracket
defined as above. 

The significance  of understanding  $Z(U(\mathfrak{g}_p))$ (as a Poisson algebra)
in regards with the derived isomorphism problem above lies in its derived invariance: if $U(\mathfrak{g})$
and $U(\mathfrak{g'})$ are derived equivalent, then $Z(U(\mathfrak{g}_{p}))\cong Z(U(\mathfrak{g'}_{p}))$
as Poisson $S/pS$-algebras for $p\gg 0.$ This easily follows from the derived invariance of the Hochschild cohomology and
the Gersenhaber bracket (see [\cite{T1} Lemma 4]).

  Recall that if $I\subset Z(U(\mathfrak{g}_p))$ is a Poisson ideal,
then the Poisson bracket induces Lie bracket on $I/I^2.$ In particular, if $\mf{m}$ is a Poisson ideal such that $Z(U(\mathfrak{g}_p))/\mf{m}=S/pS,$
then $\mf{m}/\mf{m}^2$ is a finite $S/pS$-Lie algebra. 
Therefore, the collection of isomorphisms classes of $S/pS$-Lie algebras $\mf{m}/\mf{m}^2,$
as $\mf{m}$ ranges over  maximal Poisson ideals of $Z(U(\mathfrak{g}_p))$ (so $Z(U(\mathfrak{g}_p))/\mf{m}=S/pS$)
is a derived invariant of $U(\mathfrak{g})$ (for $p\gg 0$). The significance of this derived invariant of  $\mathfrak{g}$ is highlighted by
the fact that given a Poisson ideal $\mathfrak{m}\subset Z(U(\mathfrak{g}_p))$ as above, as the key computation by Kac and Radul shows (see Lemma \ref{key} below),
 there is a canonical Lie algebra
homomorphism  $\mathfrak{g}_{p}\to \mf{m}/\mf{m}^2$ (when $\mathfrak{g}$ is an algebraic Lie algebra).
There is one distinguished such maximal Poisson ideal--the augmentation ideal 
$\mathfrak{m}(\mathfrak{g}_p)=Z(U(\mathfrak{g}_p))\cap \mathfrak{g}_{p}U(\mathfrak{g}_{p}).$

 Nex we need to recall a key computation of the Poisson bracket for restricted Lie
algebras due to Kac and Radul \cite{KR}. 
Let $R$ be a  commutative reduced ring of  characteristic $p>0.$
Let $\mathfrak{g}$ be a restricted Lie algebra over $R$
with the restricted structure map $x\to x^{[p]}, x\in \mathfrak{g}.$
 It is well-known that the map $x\to x^{p}-x^{[p]}$ induces
homomorphism of $R$-algebras $$i:\Sym (\mathfrak{g})\to Z_p(\mathfrak{g}),$$
where $Z_p(\mathfrak{g})$ is viewed as an $R$-algebra via the Frobenius map $F:R\to R.$
The homomorphism $i$ is an isomorphism when $R$ is perfect.
Recall also that
the Lie algebra bracket on $\mathfrak{g}$ defines the Kirillov-Kostant Poisson bracket on the symmetric algebra $\Sym(\mathfrak{g}).$


The following is the above mentioned key  result from \cite{KR}.

\begin{lemma}\label{key}
Let $S$ be a finitely generated integral domain over $\mathbb{Z}.$
Let $\mathfrak{g}$ be an algebraic Lie algebra over $S.$
Then $Z_p(\mathfrak{g}_p)$ is a Poisson subalgebra of $Z(U(\mathfrak{g}_p)),$ moreover the induced Poisson
bracket coincides with the negative of the Kirillov-Kostant bracket:
$$\lbrace a^p-a^{[p]}, b^p-b^{[p]}\rbrace=-([a, b]^p-[a,b]^{[p]}),\quad a\in \mathfrak{g}_p, b\in \mathfrak{g}_p.$$

\end{lemma}

We record the following simple result that illustrates usefulness of "dequantizing" $U(\mathfrak{g})$
to $\Sym(\mathfrak{g}_{\bold{k}})$ in regards with the isomorphism problem for enveloping algebras.

\begin{lemma}\label{trivial}
Let $\mathcal{L}$ be a finite dimensional Lie algebra over a field $\bold{k}.$
Let $\mf{m}$ be a Poisson ideal of $\Sym(\mathcal{L})$ such that $\Sym(\mathcal{L})/\mf{m}=\bold{k}.$
Then $\mf{m}/\mf{m}^2\cong \mathcal{L}$ as Lie algebra. In particular,
if $\mathcal{L}'$ is another Lie algebra over $\bold{k}$ such that
$\Sym(\mathcal{L})\cong \Sym(\mathcal{L'})$ as Poisson $\bold{k}$-algebras,
then $\mathcal{L}\cong\mathcal{L'}.$
\end{lemma}
\begin{proof}
Let $\mf{m}$ be a maximal Poisson ideal. Then $\mf{m}=(g-\chi(g), g\in \mathcal{L})$ for some
$\chi\in \mathcal{L}^{*}.$ It follows that $\chi$ must be a character of $\mathcal{L}$, hence
the homomorphism $g\to g-\chi(g)$ defines a Lie algebra isomorphism
$\mathcal{L}\to \mf{m}/\mf{m}^2$ as desired.

\end{proof}

\begin{proof}[Proof of Theorem \ref{iso}]

We may assume that Lie algebras $\mathfrak{g}, \mathfrak{g'}$ and the corresponding
derived isomorphism are defined over a finitely generated ring $S\subset\mathbb{C}.$
Let $Z(U(\mathfrak{g}))=S[f_1,\cdots, f_n]$ and $Z(U(\mathfrak{g'}))=S[f'_1,\cdots, f'_n].$
Thus we have an $S$-algebra isomorphism $S[f_1,\cdots, f_n]\cong S[f'_1,\cdots, f'_n]$  and a Poisson
algebra isomorphism
$Z(U(\mathfrak{g}_{\bold{k}}))\cong Z(U(\mathfrak{g'}_{\bold{k}})).$
Moreover these isomorphisms are compatible with reduction modulo $p$ maps
so that the following diagram commutes:

\[\begin{tikzcd}
S[f_1,\cdots,f_n] \arrow{r}{} \arrow[swap]{d}{} & S[f_1',\cdots, f_n'] \arrow{d}{} \\
Z(U(\mathfrak{g}_{\bold{k}})) \arrow{r}{} & Z(U(\mathfrak{g}'_{\bold{k}})).
\end{tikzcd}
\]

Let $\mf{m}$  be a maximal Poisson ideals in $Z(U(\mathfrak{g}_{\bold{k}}))$,
and let $\mf{n}\subset Z(U(\mathfrak{g'}_{\bold{k}}))$ be the corresponding maximal Poisson ideal under the above isomorphisms.
Hence we get an isomorphism of Lie algebras $\mf{m}/\mf{m}^2\cong \mf{n}/\mf{n}^2.$ 
Put 
$$\mathfrak{m}'=\mathfrak{m}\cap Z_p(\mathfrak{g}_{\bold{k}}),\quad \mathfrak{m}''=\mathfrak{m}\cap Z_{\HC}(\mathfrak{g}_{\bold{k}}), \mathfrak{m}_1=\mathfrak{m}'\cap \mathfrak{m}''.$$
Then Theorem \ref{center} implies the following Lie algebra isomorphism 
$$\mathfrak{m}/\mathfrak{m}^2\cong \mathfrak{m}'/\mathfrak{m}'^2\times_{\mf{m}_1/\mf{m}_1^2} \mf{m''}/\mf{m}''^2.$$ 
On the other hand, since the Poisson bracket vanishes on $Z_{\HC},$
we conclude that $\mf{m}/\mf{m}^2$ is a trivial central extension of the image of $\mf{m}'/\mf{m'}^2\cong \mathfrak{g}_{\bold{k}}$
and the kernel of the natural homomorphism $\mf{m'}/\mf{m'}^2\to \mf{m}/\mf{m}^2 $ is central. On the other hand, if $g\in Z(\mf{m}'/\mf{m'}^2)$, then 
$g=f^p$ with $f\in \mf{m}_1$, hence the image of $g$ in $\mf{m}/\mf{m}^2$ is 0.
Thus $\mf{m}/\mf{m}^2$ is a trivial central extension of $\mathfrak{g}_{\bold{k}}/Z(\mathfrak{g}_{\bold{k}}).$
Similarly  $\mf{n}/\mf{n}^2$ is isomorphic to
a trivial central extension  of $\mathfrak{g'}_{\bold{k}}/Z(\mathfrak{g'}_{\bold{k}}).$
Since we have the compatible isomorphism
$\bold{k}[g_1,\cdots, g_n]\cong \bold{k}[g'_1,\cdots, g'_n$] we get that 
$$\mathfrak{g}_{\bold{k}}/Z(\mathfrak{g}_{\bold{k}})\oplus V\cong \mathfrak{g}_{\bold{k}}/Z(\mathfrak{g'}_{\bold{k}})\oplus V'$$
with $V\cong V'$ being abelian $\bold{k}$-Lie algebras.
 Thus $\mathfrak{g}_{\bold{k}}/Z(\mathfrak{g}_{\bold{k}})\cong \mathfrak{g'}_{\bold{k}}/Z(\mathfrak{g'}_{\bold{k}}).$
Hence $\mathfrak{g}/Z(\mathfrak{g})\cong\mathfrak{g'}/Z(\mathfrak{g'}).$

\end{proof}



\begin{acknowledgement} 
I am  very grateful to Lewis Topley for making numerous useful suggestions.

\end{acknowledgement}


\end{document}